\documentclass[10pt]{amsart}

\usepackage{amsmath,amsthm,amssymb,amscd,
latexsym,enumerate,mathrsfs}

\usepackage[latin1]{inputenc}

\newcommand{\Z}{{\mathbb Z}}

\renewcommand{\phi}{\varphi}

\theoremstyle{plain}
    \newtheorem{thm}{Theorem}[section]
    \newtheorem{lemma}[thm]{Lemma}

\theoremstyle{definition}
    \newtheorem{defn}[thm]{Definition}
    \newtheorem{ex}[thm]{Example}

\DeclareMathOperator{\id}{id}

\title{Non-homogeneous extensions of Cantor minimal systems}

\author[R.J. Deeley, I.F. Putnam, K.R. Strung]
{Robin J. Deeley \and
Ian F. Putnam \and
Karen R. Strung}
\address{Department of Mathematics,
University of Colorado Boulder
Campus Box 395,
Boulder, CO 80309-0395, USA }
\email{robin.deeley@gmail.com}

\address{Department of Mathematics and Statistics,
University of Victoria,
Victoria, B.C., Canada V8W 3R4} 
\email{ifputnam@uvic.ca}

\address{Institute of Mathematics, Czech Academy of Sciences, \v{Z}itn\'a 25, 115 67 Prague, Czech Republic}
\email{strung@math.cas.cz}

\thanks{RJD is currently funded by NSF Grant DMS 2000057 and was previously funded by Simons Foundation Collaboration Grant for Mathematicians number 638449. KRS is currently funded by GA\v{C}R project 20-17488Y and \mbox{RVO: 67985840} and part of this work was carried out while funded by Sonata 9 NCN grant 2015/17/D/ST1/02529 and a Radboud Excellence Initiative Postdoctoral Fellowship. IFP is supported in part by an NSERC Discovery Grant.}

\date{\today}

\begin{document}

\begin{abstract}Floyd gave an example of a minimal dynamical 
system which was an extension  of an odometer and the
 fibres of the associated factor map were either 
 singletons or intervals.  Gjerde 
 and Johansen showed that the odometer could be replaced
 by any Cantor minimal system. Here, we show further that 
 the intervals can be generalized to cubes of arbitrary dimension
 and to attractors of certain iterated function systems.
 We discuss applications.
 \end{abstract}
 
\maketitle

\section{Introduction and statement of results}

We consider dynamical systems  consisting
 of a compact space, $X$, together with 
a homeomorphism, $\varphi: X \rightarrow X$.
We say that such a system is \emph{minimal} if the 
only closed sets $Y \subseteq X$ such that $\varphi(Y) = Y$ are 
$Y = X$ and $Y = \emptyset$. Equivalently, for every $x$ in $X$, its orbit,
$\{ \varphi^{n}(x) \mid n \in \Z \}$, is dense in $X$.

There are a number of examples of such systems: rotation of the 
circle through an angle which is an irrational multiple 
of $2 \pi$, odometers and certain diffeomorphisms 
of spheres of odd dimension $d \geq 3$ constructed by Fathi and Herman \cite{FatHer:Diffeo}.

All of these examples share one common feature: the spaces 
involved are 
homogeneous. There are several ways to make this more precise, 
but one simple way would be to observe that the group
 of homeomorphisms 
acts transitively on the points.

In \cite{floyd1949}, Floyd gave the first example of a minimal system 
where the space is not homogeneous in this (or an even stronger) sense. 
Floyd began with the $3^{\infty}$-odometer, $(X, \varphi)$, 
which is a minimal system with $X$ compact, metrizable, 
totally disconnected and without isolated points. Any two such 
spaces are homeomorphic and we refer to such a space as a 
\emph{Cantor
set}. Floyd then constructed another minimal
system, $(\tilde{X}, \tilde{\varphi})$, together with a continuous 
surjection $ \pi: \tilde{X} \rightarrow X$ 
satisfying $\pi \circ \tilde{\varphi} = \varphi \circ \pi$. 
In general, we refer to such a map as a \emph{factor map},
 we say 
that $(X, \varphi)$ is a \emph{factor} of $(\tilde{X}, \tilde{\varphi})$ and 
that $(\tilde{X}, \tilde{\varphi})$ is an 
\emph{extension} of $(X, \varphi)$. In Floyd's example, 
some points $x$ in $X$ have $\pi^{-1}\{ x \}$ homeomorphic
to the unit interval, $[0, 1]$, while for others, it is a single
 point.
It is then quite easy to see, using the fact that $X$ is 
totally disconnected, 
 that the space $\tilde{X}$ has some
connected components which are single points and 
some homeomorphic to
the interval.

This example has been generalized in several ways (for example \cite{auslander1959, MR956049, FloGjeJohSys, HadJoh:AuslanderSys}). In Floyd's 
example, the points $x$ with $\pi^{-1}\{ x \}$ infinite all
lie in a single orbit. Haddad and Johnson in \cite{HadJoh:AuslanderSys} showed that the set
of such $x$ could be much larger and even have positive measure
under the unique invariant measure for $(X, \varphi)$.
More importantly for our purposes, Gjerde and Johansen \cite{FloGjeJohSys} showed
that the $3^{\infty}$-odometer could be replaced with any 
minimal system, $(X, \varphi)$, with $X$ a Cantor set. Their
 principal tool was the Bratteli--Vershik model 
 for such systems \cite[Chapters 4 and 5]{Put:BookMinCantorSys}.
 We will describe this in more detail in Section \ref{Sec:ConstProofs}.
 
Our aim here is to show that the interval, $[0,1]$,  appearing as 
$\pi^{-1}\{ x\}$, can be replaced by more complicated spaces. We are  
particularly interested in the case of the 
$n$-dimensional cube (that is, $[0,1]^{n}$), for any 
positive integer $n$. 

Although it is natural to generalize to more complicated spaces,
 let us explain briefly why we want such a result 
 in the specific case of $[0,1]^n$.
 The Elliott program aims to show that 
 a broad class of $C^{*}$-algebras may be classified up 
 to isomorphism by their $K$-theory \cite{Ell:classprob}. One very useful
 way of constructing $C^{*}$-algebras is via groupoids \cite{MR584266} and
  it becomes a natural question: which $C^{*}$-algebras in 
  the Elliott scheme
 can be realized via a groupoid construction? 
 In view of the classification results themselves, this amounts
 to constructing groupoids whose associated $C^{*}$-algebras
 are classifiable and have some prescribed $K$-theory.
If one begins with a minimal action of the integers
on a Cantor set, it is known that the $K_{0}$-group is a simple 
acyclic dimension
group and $K_{1}$ is the integers \cite{MR1194074}.
 Moreover, any such
$K$-theory can be realized from such a system \cite{MR1194074}. 

In another direction, if one takes a minimal action, $\varphi$, 
 of the 
integers on some space $X$ and considers a closed, 
non-empty  subset $Y \subseteq X$ such that $Y$ meets each orbit
at most once, one can construct the associated ``orbit-breaking groupoid'': 
the equivalence relation where the classes are either the original orbits of 
$\varphi$ which do not meet $Y$ 
or the half-orbits, split at $Y$. The change in $K$-theory passing from the 
crossed product $C^{*}$-algebra to the orbit-breaking subalgebra
can be computed, essentially in terms of the $K$-theory
 of the space $Y$ (see \cite{Put:Excision} for details).

Marrying these two ideas would seem to generate many interesting
groupoids, except that the choices for $K^{*}(Y)$, where 
$Y$ is a closed subset of the Cantor set, are very limited. Here, we 
would like to
replace the dynamics $(X, \varphi)$ with $(\tilde{X}, \tilde{\varphi})$, 
without changing the associated $K$-theory, but allowing us
to find more interesting spaces $Y$ inside of 
$[0, 1]^{n} \cong \pi^{-1} \{ x\}$. These $C^*$-algebraic
 applications can be found in \cite{DPS:minDSKth}.

Our construction and proof follow those of Gjerde and
 Johansen in \cite{FloGjeJohSys} quite closely and, in turn, their proof is
 quite similar to Floyd's original one \cite{floyd1949}. 
 One added feature here is that we use the framework of 
 iterated function systems, as this allows us to replace
  the interval, $[0,1]$, with the more complicated spaces. 

Following usual conventions (see for example \cite{Hut:FraSelfSim}) an 
\emph{iterated function system}
consists of a metric space, $(C, d_{C})$, and 
 $\mathcal{F}$,  a finite collection
of maps $f: C \rightarrow C$ with the property that
there is a constant $0 < \lambda < 1$
 such that  $d_{C}(f(x), f(y)) \leq \lambda d_{C}(x,y)$, 
for all $x, y$ in $C$ and $f$ in $\mathcal{F}$.
 In particular, each map is 
continuous.
We will require a few extra properties.

\begin{defn}
\label{1-10}
Let $(C,d_{C}, \mathcal{F})$ be an iterated function system.
We say it is \emph{compact}
 if the metric space $(C, d_{C})$ is compact. We also say it is 
 \emph{invertible} if
\begin{enumerate}
\item each $f$ in $\mathcal{F}$ is injective, and
\item $\cup_{f \in \mathcal{F}} f(C) = C.$
\end{enumerate}
\end{defn}

The term ``invertible'' is meant  to indicate that each map
$f$ in $\mathcal{F}$ has an inverse, $f^{-1}:f(C) \rightarrow C$.
It is not ideal as it does not rule out the possibility that
the images of the various $f$'s  overlap.

Of course, the restriction that each map is injective
is quite important. On the other hand, it is well-known that
any compact iterated function system has a fixed point set
and the restriction of the maps to this set will satisfy 
the invariance condition  \cite[Section 3]{Hut:FraSelfSim}.

We list several simple examples of relevant iterated function systems. The 
first is the one originally used by Floyd \cite{floyd1949} along with the subsequent examples \cite{auslander1959, MR956049, FloGjeJohSys, HadJoh:AuslanderSys}.

\begin{ex}
\label{1-20}
Let $C= [0, 1]$, $f_{i}(x) = 2^{-1}(x + i)$ 
for $x$ in $[0, 1]$ and $i=0,1$, 
and $\mathcal{F} = \{ f_{0}, f_{1} \}$.
\end{ex}

The next example is a fairly simple generalization
 of the last, but it is important as this is
  the example
we need in our applications in \cite{DPS:minDSKth}.

\begin{ex}
\label{1-30}
Let $n$ be any positive integer, 
 $C=  [0, 1]^{n}$, $ f_{\delta}(x) = 2^{-1}( x + \delta)$, for 
each $x$ in $[0, 1]^{n}, \delta \in \{ 0, 1\}^{n}$ and 
$\mathcal{F} = \{ f_{\delta} \mid \delta \in \{ 0, 1 \}^{n} \}$.
\end{ex}

\begin{ex}
\label{1-40}
A minor variation on the last example would be to 
 use  instead $ f_{\delta}(x) = 3^{-1}( x + \delta)$, 
for 
each $x$ in $[0, 1]^{n}, \delta \in \{ 0, 1, 2 \}^{n}$.
 On the other hand, if we
instead let 
$\mathcal{F} =  \{ f_{\delta} \mid \delta \in \{ 0,  2\} \}$
 when $n=1$, 
or $\mathcal{F} =  
\{ f_{\delta} \mid \delta \in \{ 0, 1, 2\}^{2}, \delta \neq (1,1) \}$
for $n=2$,
this now fails the invariance condition of our definition. 
As mentioned above,
 standard results on iterated function 
systems show that $C$ contains  a unique closed set 
and restricting our maps
to that set then satisfies all the desired conditions.
Notice that when $n=1$, the set in 
question is the Cantor ternary set, 
while for $n=2$, it is the Sierpinski carpet \cite{Sie:Carpet}.
\end{ex}

Our main result is the following.

\begin{thm}
\label{1-50}
Let $(C, d_{C}, \mathcal{F})$ be a compact, invertible 
iterated function system and let 
$(X, \varphi)$ be a minimal 
homeomorphism of the Cantor set.
There exists a minimal extension, $(\tilde{X}, \tilde{\varphi})$ 
of $(X, \varphi)$ with factor map 
$\pi: (\tilde{X}, \tilde{\varphi}) \rightarrow (X, \varphi)$
such that, for each $x$ in $X$, $\pi^{-1}\{ x \}$ is a 
single point or is homeomorphic to $C$. 
Moreover, both possibilities occur.
\end{thm}

\begin{thm}
\label{1-60}
Let $(C, d_{C}, \mathcal{F})$ be a compact, invertible 
iterated function system and let 
$(X, \varphi)$ be a minimal 
homeomorphism of the Cantor set.
If $C$ is contractible, then the minimal extension $(\tilde{X}, \varphi)$ and the factor map
$\pi: (\tilde{X}, \tilde{\varphi}) \rightarrow (X, \varphi)$ 
may be chosen so that 
\[
\pi^{*}: K^{*}(X) \rightarrow K^{*}(\tilde{X})
\]
is an isomorphism and so that 
$\pi$ induces a bijection between the respective sets
of invariant measures.
\end{thm}

\section{The construction and proofs} \label{Sec:ConstProofs}

Just as for Gjerde and Johansen  \cite{FloGjeJohSys}, we make critical use
of the Bratteli--Vershik model for minimal systems on the Cantor set \cite{Put:BookMinCantorSys}. Briefly, the Bratteli-Vershik model takes
some simple combinatorial data (an ordered Bratteli diagram)
and produces a minimal homeomorphism of the Cantor set. In fact, 
every minimal homeomorphism of the Cantor set is produced in 
this way.
A standard reference for Cantor minimal system is \cite{Put:BookMinCantorSys}, in particular see Chapters 4 and 5 for more on the Bratteli--Vershik model. 

We begin with a Bratteli diagram, $(V, E)$, consisting
of a vertex set $V$ written as a disjoint union of finite, non-empty
sets $V_{n}, n \geq 0$, with $V_{0} = \{ v_{0} \}$, and  an edge set
written as a disjoint union of finite, non-empty
sets $E_{n}, n \geq 1$. Each edge $e$ in $E_{n}$ has a 
source, $s(e)$, in $V_{n-1}$ and range, $r(e)$, in $V_{n}$.
We may assume (see \cite{Put:BookMinCantorSys}) that our diagram has 
\emph{full edge connections}, that is, every pair of vertices from 
adjacent levels is connected by at least one edge. We define the 
space $X_{E}$ to consist of all infinite paths in the 
diagram, beginning at $v_{0}$. That is, a point
 $x = (x_{1}, x_{2}, \ldots )$, 
$x_{n} \in E_{n}, r(x_{n}) = s(x_{n+1})$. This space
 is endowed with the metric
 \[
 d_{E}(x,y) = \inf \{ 2^{-n} \mid n \geq 0, 
 x_{i} = y_{i}, 1\leq i \leq n \}.
 \]
 
 In addition, we may assume that the edge set $E$ is endowed with an order 
 such that two edges $e,f$ are comparable if and only if $r(e) =r(f)$.
The set of maximal edges and the set of minimal edges each form a tree
and we  assume that our diagram is \emph{properly ordered}, meaning that
each contains exactly one infinite path. Two finite paths from $v_{0}$
to $V_{n}$ can be compared if they have the same range vertex by
 using a right-to-left lexicographic order. 
 Infinite paths may be compared in a similar way: two
  paths are \emph{cofinal} if they differ in only finitely many entries and
  can be compared using a right-to-left lexicographic order. 
  The Bratteli--Vershik map, $\varphi_{E}$, takes an infinite path to its
  successor, and the unique infinite path with all  edges maximal to
  the unique infinite path with all  edges minimal. The system $(X_{E}, \varphi_{E})$
  is a minimal Cantor system (provided $X_{E}$ is infinite).
   Moreover, every minimal Cantor system
  is topologically conjugate to a Bratteli--Vershik system.
In view of this, we may assume that $(X, \varphi) = (X_{E}, \varphi_{E})$, 
for some properly ordered Bratteli diagram, $(V, E)$.

We note that there is an (essentially) unique ordered Bratteli diagram with
$\# V_{n} =1$ and $\# E_{n} = 3$, for all $n \geq 1$, and the associated
Bratteli--Vershik map is the $3^{\infty}$-odometer considered by
Floyd. More generally, an odometer is any system with
$\# V_{n} =1$, for all $n \geq1$ (see for example \cite[Chapter 11 Section 8]{Put:BookMinCantorSys}).

Recall that, in addition to the Cantor minimal system, we also
 have $(C, d_{C}, \mathcal{F})$, which is a compact, invertible 
iterated function system. Our final ingredient involves this system. 
To each 
edge $e$ in $E$, we assign a function, denoted $f_{e}$, 
in $\mathcal{F} \cup \{ \id_{C} \}$, 
where $id_{C}$ is the identity function on $C$. 
We assume that this assignment satisfies the following three
conditions:
\begin{enumerate}
\item if $e$ is either maximal or minimal, then $f_{e} \neq \id_{C}$,
\item
  for every $v$ in $V$, we have 
 $\cup_{s(e) =v, f_{e} \neq \id_{C}} f_{e}(C) = C,$
\item the set 
$\{ e \in E \mid f_{e} = \id_{C} \}$ contains an infinite path.
\end{enumerate}

Let us first mention that the following weaker third 
condition 
will suffice: there exists an infinite path $(e_{1}, e_{2}, \ldots)$ such 
that $f_{e_{n}} = \id_{C}$, for infinitely many $n$.
 Secondly, the fact that we can find a properly
 ordered Bratteli diagram satisfying these can be seen as follows.
 First, take any minimal Bratteli--Vershik 
 $(X_{E}, \varphi_{E})$ with $X_{E}$ infinite.
  It can be telescoped until every pair of vertices
 at adjacent levels have at least $\# \mathcal{F} +1$ edges (see
  \cite[page 94]{FloGjeJohSys} or \cite[page 22--23]{Put:BookMinCantorSys}).
 Select some infinite path, $(e_{1}, e_{2}, \ldots )$,
 which avoids all maximal and minimal edges and set
 $f_{e_{n}}= \id_{C}$ for all $n \geq 1$. 
 Then $f$ may be chosen so that it maps the  set
 $s^{-1}\{ v\} - \{ e_{1},  e_{2}, \ldots \}$
  surjectively to $\mathcal{F}$, for every $v $ in $V$.

In the construction of Gjerde and Johansen \cite{FloGjeJohSys}, the 
edges $e$ with $f_{e} = \id_{C}$
 form a single infinite path. 
On the other hand, 
allowing more general subdiagrams, we capture examples such 
as those given in \cite{HadJoh:AuslanderSys}.

We now  construct the system $(\tilde{X}, \tilde{\varphi})$. We endow 
$X_{E} \times C$ with 
the metric
\[
d\left( (x,c), (y, d) \right) = \max \{ d_{E}(x,y), d_{C}(c,d) \},
\]
for $x,y$ in $X_{E}$ and $c, d $ in $C$.
For each $n \geq 1$, define
\[
\tilde{X}_{n} = \{ (x, c) \in X_{E} \times C \mid 
c \in f_{x_{1}} \circ \cdots \circ f_{x_{n}}(C) \}.
\]
It is immediate that each $\tilde{X}_{n}$ is 
closed and non-empty   and that $\tilde{X}_{n} \supseteq \tilde{X}_{n+1}$.
 We let $\tilde{X} = \cap_{n \geq 1} \tilde{X}_{n}$, 
 which is also closed and non-empty.

The quotient map $\pi: \tilde{X} \rightarrow X_{E}$ is defined by
$\pi(x,c) = x$.

\begin{lemma}
\label{2-40}
Suppose that $x \in X_{E}$. Then exactly one of the following hold:
\begin{enumerate}
\item Type 1: for infinitely many $n \geq 1$, $f_{x_{n}} \neq \id_{C}$,
\item
Type 2: there exists $n \geq 1$ such that 
$f_{x_{i}} = \id_{C}$, for all $i \geq n$.
\end{enumerate}
Moreover, in the Type 1 case,  $ \pi^{-1}\{ x \} $ consists of a single point, which we 
 denote by $(x, c_{x})$. 
In the Type 2 case, we have 
 \[
 \pi^{-1}\{ x \} = \{ x \} \times 
 f_{x_{1}} \circ \cdots \circ f_{x_{n}}(C)
 \]
 which is homeomorphic to $C$, since each $f_{e}$ is continuous and injective.
\end{lemma}

\begin{proof}
The set of $n$ such that $f_{x_{n}} = \id_{C}$ is either finite or
infinite, hence the two conditions are mutually 
exclusive and the only possibilities.

 In the first case,  we have 
 \[
 {\rm diam}( f_{x_{1}} \circ \cdots \circ f_{x_{n}}(C)) \leq 
  \lambda^{m}  {\rm diam}(C),
 \]
 where $m$ is the number of $1 \leq l \leq n$ with $f_{x_{l}} \neq \id_{C}$.
 The conclusion follows.

In the second case, it is clear that
\[
\left( \{ x \} \times C \right) \cap \tilde{X}_{i} =
 \{ x \} \times f_{x_{1}} \circ \cdots \circ f_{x_{n}}(C),
 \]
 for any $i \geq n$ and now, taking the intersection over all $i$, we obtain
 the desired conclusion.
\end{proof}

We are now prepared to define our self-map of $\tilde{X}$ and show it is a minimal homeomorphism.
Let $(x,c)$ be in $\tilde{X}$. We consider three cases separately. 
First, we 
 assume that $x$ is 
Type 1 and that $x \neq x^{\max}$. Then $x$ contains a non-maximal edge
and we let $n$ be the first such edge. Thus, we have 
$\varphi_{E}(x) = (y_{1}, \ldots, y_{n}, x_{n+1}, \ldots )$, 
for some path $y_{1}, \ldots, y_{n}$ with 
$r(y_{n}) = r(x_{n})$.
 The fact that $x$ is Type 1 means that $x$
  uniquely determines $c = c_{x}$ and 
we define $\tilde{\varphi}(x, c) = (\varphi_{E}(x), c_{\varphi_{E}(x)})$.
To see this is well-defined, it suffices to note that, as 
$x$ is Type 1 and  $\varphi_{E}(x)$  differs from $x$ in only finitely many 
entries, $\varphi_{E}(x)$ is also Type 1 and 
$c_{\varphi_{E}(x)}$ is well-defined.

Secondly, suppose that $x = x^{\max}$. It follows 
that, for all $n$,   $x_{n}$ is maximal and hence
 $f_{x_{n}} \neq \id_{C}$. 
So 
$x^{\max}$ is also Type 1 and $c$ is once again uniquely
determined as $c= c_{x}$. The same argument shows that 
$x^{\min}$ is Type 1 as well. We define
$\tilde{\varphi}(x^{\max}, c_{x^{\max}}) = (x^{\min}, c_{x^{\min}})$. 

Finally, we consider the case that $x$ is Type 2. 
In particular, this implies that $x$ is not
equal to $x^{\max}$. Let $n$ be as in the definition of 
Type 2. (Notice that such an $n$ is not unique: we deal with 
this issue shortly.)
 As mentioned, $x_{n}$ is not maximal. Hence
$\varphi_{E}(x) = (y_{1}, \ldots, y_{n}, x_{n+1}, \ldots)$, for some
path $y_{1}, \ldots, y_{n}$ with 
$r(y_{n}) = r(x_{n})$. 
(It is worth noting that it is possible that $x_{n}=y_{n}$.)
Since $c$ is in 
$f_{x_{1}} \circ \cdots \circ f_{x_{n}}(C)$ we can 
 define
\[
\tilde{\varphi}(x,c) = (\varphi_{E}(x), 
f_{y_{1}} \circ \cdots \circ f_{y_{n}} 
\circ f_{x_{n}}^{-1} \circ \cdots \circ f_{x_{1}}^{-1}(c)).
\]

As mentioned in the previous paragraph the choice of $n$
 is not unique and hence we need to check that the above
  definition is independent of the choice of 
$n$. This follows from the observation that 
if  $n'$ is another such integer, we
may assume that $n' > n$, and then all the maps
$f_{x_{n+1}}, \ldots, f_{x_{n'}}$ are all equal to $id_{C}$.

The proof that $\tilde{\varphi}$ is bijective is as follows:
by simply reversing the order on the edge set, the construction 
will yield another map which is easily seen to be
the inverse of $\tilde{\varphi}$.

We claim that $\tilde{\varphi}$ is continuous. 
To show this, 
it will be convenient to define, for any path 
$p=( p_{1}, \ldots, p_{n})$ in $(V,E)$ from $v_{0}$ 
to $V_{n}$, the sets
\[
\tilde{X}_{n}(p) =
\{ x \in X_{E}  \mid (x_{1}, \ldots, x_{n})=p \} \times  
f_{p_{1}} \circ \cdots \circ f_{p_{n}}(C).
\]
It is an easy matter to check that $\tilde{X}_{n}(p)$
  is a closed subset 
of $\tilde{X}_{n}$, that for $p \neq q$ of length $n$, 
$\tilde{X}_{n}(p)$ and  $\tilde{X}_{n}(q)$ are disjoint, and that 
the union
of all such sets over paths of length $p$ is exactly $\tilde{X}_{n}$.
From this, it follows that each such set is also an open subset
 $\tilde{X}_{n}$. It also follows that 
 $\tilde{X} \cap \tilde{X}_{n}(p)$ is clopen in $\tilde{X}$.

Let $p$ be any path in $E$ from $v_{0}$ to $V_{n}$ which is not 
maximal. Let $q$ be its successor among such paths.
Define a map $\psi: \tilde{X}_{n}(p) \rightarrow 
\tilde{X}_{n}(q) $ by 
\[
\psi(z, c) = (\varphi_{E}(z), f_{q_{1}} \circ \cdots \circ f_{q_{n}} 
\circ f_{p_{n}}^{-1} \circ \cdots \circ f_{p_{1}}^{-1}(c)),
\]
for $(z,c)$ in $\tilde{X}_{n}(p)$. Observe 
that this is clearly a homeomorphism.

 We will show 
$\psi|_{\tilde{X} \cap \tilde{X}_{n}(p)} =
\tilde{\varphi}|_{\tilde{X} \cap \tilde{X}_{n}(p)}$.
From this and the fact that  $\tilde{X} \cap \tilde{X}_{n}(p)$
is clopen in $\tilde{X}$, 
it follows that $\tilde{\varphi}$ is continuous on 
$\tilde{X} \cap \tilde{X}_{n}(p)$.

First, suppose that $z$ is a Type 2 point in 
$\tilde{X} \cap \tilde{X}_{n}(p)$. Choose $m$ such that 
$f_{z_{i}} = \id_{C}$, for all $i \geq m$,
 Without loss of generality, we may assume that
$m> n$. We know that
 $\varphi_{E}(z) = (q_{1}, \ldots, q_{n}, z_{n+1}, \ldots )$ and hence, 
 for any $c$ in $C$, 
 \begin{eqnarray*}
 \tilde{\varphi}(z,c) & =  & ((q_{1}, \ldots, q_{n}, z_{n+1}, \ldots ),
 f_{q_{1}} \circ \cdots \circ f_{q_{n}} \circ f_{z_{n+1}} \circ 
\cdots \circ  f_{z_{m}} \\
  &  &  \circ f_{z_{m}}^{-1} \circ \cdots 
f_{z_{n+1}}^{-1} \circ f_{p_{n}}^{-1}  \circ \cdots \circ 
f_{p_{1}}^{-1}(c) ) \\
  & =  & ((q_{1}, \ldots, q_{n}, z_{n+1}, \ldots ),
 f_{q_{1}} \circ \cdots 
 \circ f_{q_{n}} \circ f_{p_{n}}^{-1}  \circ \cdots \circ 
f_{p_{1}}^{-1}(c)  ) \\
  &  =  & \psi(z, c).
  \end{eqnarray*}
  
  Now, we consider a point $z$ of Type 1 in
   $\tilde{X} \cap \tilde{X}_{n}(p)$.
  The same argument as above shows that, for any $m > n$, we have
  \[
  \psi( \tilde{X} \cap \tilde{X}_{n}(p, z_{n+1}, \ldots, z_{m})  )
  =  
  \tilde{\varphi}( \tilde{X} \cap \tilde{X}_{n}(p, z_{n+1}, \ldots, z_{m})). 
  \]
 The point $\psi(z, c_{z})$ is the unique point   which lies in the left-hand
 side for every $m > n$, while $\tilde{\varphi}(z, c_{z})$ is the unique
 point that lies in the right-hand side, for every $m > n$. Hence, we
 conclude they are equal.
 
We have now shown that $\tilde{\varphi}$ is continuous on every set 
$\tilde{X}_{n}(p)$, where $p$ is a finite path which is not maximal. But 
such sets, allowing both $n$ and $p$ to vary,
contain every point of $\tilde{X}$, except $(x^{\max}, c_{x^{\max}})$. 
It follows from general topological arguments using the facts that 
$\tilde{\varphi}$ is a bijection and that $\tilde{X}$ is
compact, that $\tilde{\varphi}$ is continuous everywhere.

Finally, we need to show that $(\tilde{X}, \tilde{\varphi})$ is minimal. 
Let $(x, c)$ and $(y, d)$ be in $\tilde{X}$ and let $\epsilon > 0$.
It suffices for us to find a point $(z, e)$ in the orbit of
$(x, c)$ within distance $\epsilon$ of $(y,d)$.

We first consider the case that $y$ is of Type 1. 
We choose $n$ sufficiently
large so that $2^{-n} < \epsilon$ and so that, if we let 
$m$ be 
the number of $1 \leq l \leq n$ such that $f_{y_{l}} \neq \id_{C}$, 
then $\lambda^{m} {\rm diam}(C) < \epsilon$. It follows that
$\tilde{X} \cap \tilde{X}_{n}(y_{1}, \ldots, y_{n})$ is a clopen set
containing $(y,d)$ of diameter less than $\epsilon$.

 We define $z$ in $X_{E}$ as follows: $z_{i} = y_{i}$ for all
 $1 \leq i \leq n$, $z_{n+1}$ is any edge in $E_{n+1}$ 
 with $s(z_{n+1}) = r(y_{n})$, and $r(z_{n+1}) = s(x_{n+2})$ and 
 $z_{i} = x_{i}$, for any $i \geq n+2$. It follows at once that
$z$ and $x$ are cofinal.
 Hence, there is an integer $k$ such that $\varphi_{E}^{k}(x) = z$.
It is then clear that 
  $\tilde{\varphi}^{k}(x,c)$ is in
   $\tilde{X} \cap \tilde{X}_{n}(y_{1}, \ldots, y_{n})$
    and hence within $\epsilon$ of $(y,d)$.
    
    Now we consider the case that $y$ is of Type 2.  First,
     choose
    $m$ sufficiently large so that $2^{-m} < \epsilon$ and 
    so that  $f_{y_{n}} = \id_{C}$, for all
    $n > m$. Define $z_{i} = y_{i}$, for all $ 1 \leq i \leq m$. 
    Observe that $d $ is in $f_{y_{1}} \circ \cdots \circ f_{y_{m}}(C)$
    and we let $d_{m}= f_{y_{1}}^{-1} \circ \cdots \circ f_{y_{m}}^{-1}(d)$.
   Choose $n > m$ such that 
    $\lambda^{n-m} {\rm diam}(C) < \epsilon$. For each $ m < l \leq n$, we  define 
    $z_{l}$ and $d_{l}$ inductively,  using our second hypothesis
     on the assignment $e \rightarrow f_{e}$, so that $s(z_{l}) = r(z_{l-1})$,
     $f_{z_{l}} \neq \id_{C}$
    and $f_{z_{l}}(d_{l}) = d_{l-1}$. 
    This obviously implies that 
    $f_{z_{m+1}} \circ \cdots \circ f_{z_{n}}(d_{n})=d_{m}$.
    We define $z_{n+1}$ to be any edge with $s(z_{n+1}) = r(z_{n})$ and 
    $r(z_{n+1}) = s(x_{n+2})$ and then $z_{l} = x_{l}$, for 
    $l \geq n+2$. This means that $z$ is a  path  in $X_{E}$ 
    which is cofinal with $x$. Hence, there is an integer $k$ such that
    $\varphi^{k}(x)=z$. 
    
    We claim that $\tilde{\varphi}^{k}(x,c)$, which we denote $(z, e)$,
     is within $\epsilon $ of
    $(y,d)$. First, as $\varphi_{E}^{k}(x)_{i} = z_{i} = y_{i}$, for
    all $1 \leq i \leq n$, we have 
    $d( \varphi_{E}^{k}(x), y) < 2^{-n} \leq 2^{-m} < \epsilon$.
    This also means that $e$ is in 
    $f_{z_{1}} \circ \cdots \circ f_{z_{n}}(C)$. On the other hand, we
    know that 
    \begin{eqnarray*}
    d & = & f_{y_{1}} \circ \cdots \circ f_{y_{m}}(d_{m}) \\
       & = & f_{z_{1}} \circ \cdots \circ f_{z_{m}}(d_{m}) \\
         & = & f_{z_{1}} \circ \cdots \circ f_{z_{n}}(d_{n}) \\
         &  \in &  f_{z_{1}} \circ \cdots \circ f_{z_{n}}(C).
         \end{eqnarray*}
         We also know that, since $f_{z_{l}} \neq \id_{C}$, for 
         $ m < l \leq n$, we have
         \begin{eqnarray*}
         {\rm diam}(  f_{z_{1}} \circ \cdots \circ f_{z_{n}}(C)) & \leq  &
     {\rm diam}(  f_{z_{m+1}} \circ \cdots \circ f_{z_{n}}(C)) \\
       &  \leq   &  \lambda^{n-m}
     {\rm diam}(C) \\
       & < & \epsilon.
     \end{eqnarray*}
    This completes the proof of Theorem \ref{1-50}.
    
    We now turn to the proof of Theorem \ref{1-60}.
    Let $\pi_{n}: \tilde{X}_{n} \rightarrow X_{E}$ be the obvious extension of 
    $\pi$: simply projecting onto the first factor. 
    For any path $p$, from $v_{0}$ to $V_{n}$: $\pi_{n}$ 
    maps $\tilde{X}_{n}(p)$ to $X(p) = \{ x \in X_{E} 
    \mid x_{i} = p_{i}, 1 \leq i \leq n \}$.
    As we assume $C$ is contractible, so is
    $f_{p_{1}} \circ \cdots \circ f_{p_{n}}(C) $ and this map
    induces an isomorphism on $K$-theory. Taking the union over all
    paths $p$ of length $n$, we see that $\pi_{n}$ induces an 
    isomorphism from
     $K^{*}(\tilde{X}_{n} ) \cong \oplus_{p} K^{*}(\tilde{X}_{n}(p) ) $ 
     to $K^{*}(X_{E} ) \cong \oplus_{p} K^{*}(X(p) )$.
    
    As $\tilde{X} = \cap_{n} \tilde{X}_{n} $, it is also the 
    inverse limit of 
    \[
    \tilde{X}_{1  } \leftarrow   \tilde{X}_{2 } \leftarrow   \tilde{X}_{2 } 
     \leftarrow   \cdots
     \]
     where the maps are the inclusions. As $K$-theory is continuous, the 
     conclusion follows.
     
     If we choose our assignment $e \rightarrow f_{e}$ so that 
     the edges which are assigned $\id_{C}$ form a single 
     infinite path, then $\pi$ is a bijection, except on a single orbit of
     $\varphi_{E}$. This orbit has measure zero under
     any invariant probability measure on $X$, so it lifts to
     a unique $\tilde{\varphi}$-invariant measure on $\tilde{X}$.

\subsection*{Acknowledgements}  
The authors thank the Banff International Research Station and the organisers of the workshop %uture Targets in the Classification Program for Amenable $\mathrm{C}^*$-
Future Targets in the Classification Program for Amenable $\mathrm{C}^*$-Algebras where this project was initiated.
Thanks also to the Department of Mathematics and Statistics at the University of Victoria and the Department of Mathematics of the University of Colorado Boulder for research visits facilitating this collaboration. Work on the project was also facilitated by the Lorentz Center where the first and third authors attended a conference on Cuntz--Pimsner algebras in June 2018. The authors are also grateful to the referee for a thorough reading of an earlier version of the paper and many helpful suggestions.


\begin{thebibliography}{10}

\bibitem{auslander1959}
Joseph Auslander.
\newblock Mean-{L}-stable systems.
\newblock {\em Illinois J. Math.}, 3(4):566--579, 12 1959.

\bibitem{MR956049}
Joseph Auslander.
\newblock {\em Minimal flows and their extensions}, volume 153 of {\em
  North-Holland Mathematics Studies}.
\newblock North-Holland Publishing Co., Amsterdam, 1988.
\newblock Notas de Matem\'{a}tica [Mathematical Notes], 122.

\bibitem{DPS:minDSKth}
Robin~J. Deeley, Ian~F. Putnam, and Karen~R. Strung.
\newblock Constructions in minimal amenable dynamics and applications to the
  classification of {$\mathrm{C}^*$}-algebras.
\newblock {P}reprint, 2019.

\bibitem{Ell:classprob}
George~A. Elliott.
\newblock {The classification problem for amenable $C^*$-algebras}.
\newblock In {\em Proceedings of the International Congress of Mathematicians,
  Z\"urich, 1994}, volume 1,2, pages 922--932. Birkh\"auser, Basel, 1995.

\bibitem{FatHer:Diffeo}
Albert Fathi and Michael~R. Herman.
\newblock Existence de diff\'{e}omorphismes minimaux.
\newblock pages 37--59. Ast\'{e}risque, No. 49, 1977.

\bibitem{floyd1949}
Edwin~E. Floyd.
\newblock A nonhomogeneous minimal set.
\newblock {\em Bull. Amer. Math. Soc.}, 55(10):957--960, 10 1949.

\bibitem{FloGjeJohSys}
Richard Gjerde and {\O}rjan Johansen.
\newblock {$\mathrm{C}^*$}-algebras associated to non-homogeneous minimal
  systems and their {K}-theory.
\newblock {\em Math. Scand.}, 85(1):87--104, 1999.

\bibitem{HadJoh:AuslanderSys}
Kamel~N. Haddad and Aimee S.~A. Johnson.
\newblock Auslander systems.
\newblock {\em Proc. Amer. Math. Soc.}, 125(7):2161--2170, 1997.

\bibitem{MR1194074}
Richard~H. Herman, Ian~F. Putnam, and Christian~F. Skau.
\newblock Ordered {B}ratteli diagrams, dimension groups and topological
  dynamics.
\newblock {\em Internat. J. Math.}, 3(6):827--864, 1992.

\bibitem{Hut:FraSelfSim}
John~E. Hutchinson.
\newblock Fractals and self-similarity.
\newblock {\em Indiana Univ. Math. J.}, 30(5):713--747, 1981.

\bibitem{Put:Excision}
Ian~F. Putnam.
\newblock {An excision theorem for the {$K$}-theory of {$C^*$}-algebras}.
\newblock {\em {J. Operator Theory}}, {38}({1}):{151--171}, {1997}.

\bibitem{Put:BookMinCantorSys}
Ian~F. Putnam.
\newblock {\em Cantor minimal systems}, volume~70 of {\em University Lecture
  Series}.
\newblock American Mathematical Society, Providence, RI, 2018.

\bibitem{MR584266}
Jean Renault.
\newblock {\em A groupoid approach to {$C^{\ast} $}-algebras}, volume 793 of
  {\em Lecture Notes in Mathematics}.
\newblock Springer, Berlin, 1980.

\bibitem{Sie:Carpet}
Wac{\l}aw Sierpi\'{n}ski.
\newblock {Sur une courbe \textit{cantor}ienne qui contient une image
  biunivoque et continue de toute courbe donn\'ee.}
\newblock {\em {C. R. Acad. Sci., Paris}}, 162:629--632, 1916.

\end{thebibliography}
\end{document}